\documentclass[12pt,reqno]{amsart}
\usepackage{amsmath,amssymb}
\usepackage{cite}

 \makeatletter
 \evensidemargin  \oddsidemargin
 \makeatother

\newtheorem{theorem}{Theorem}
\theoremstyle{plain}

\newtheorem{corollary}{Corollary}

\newtheorem{example}{Example}

\newtheorem{lemma}{Lemma}

\numberwithin{equation}{section}

 \numberwithin{equation}{section}

\begin{document}

\title[Some Gr\"{u}ss type inequalities...]{Some Gr\"{u}ss type inequalities for Fr\'echet differentiable mappings}
\author[A. G. Ghazanfari]{T. Teimouri-Azadbakht$^{1}$, A. G. Ghazanfari$^{2*}$}

\address{$^{1,2}$Department of Mathematics\\
Lorestan University, P.O. Box 465\\
Khoramabad, Iran.}

\thanks{*Corresponding author}

\email{$^{1}$t.azadbakhat88@gmail.com, $^{2}$ghazanfari.a@lu.ac.ir}



\begin{abstract}\noindent
 Let $X$ be a Hilbert $C^{*}$-module on $C^{*}$-algebra $A$ and $p\in A$. We denote by $D_{p}(A ,X)$ the set of all
 continuous functions $f : A \rightarrow X$, which are Fr\'{e}chet differentiable on a open neighborhood $U$ of $p$.
Then, we introduce some generalized semi-inner products on $D_{p}(A ,X)$, and using them some Gr\"{u}ss type inequalities in semi-inner product
$C^*$-module $D_{p}(A ,X)$ and $D_{p}(A ,X^n)$ are established.
\end{abstract}
\maketitle

\section{Introduction}\noindent

Let $A,X$ be two normed vector spaces over $\mathbb{K}(\mathbb{K} =\mathbb{C}, \mathbb{R})$, we recall that a function
$f:A \rightarrow X$ is Fr\'echet differentiable in $p \in A$, if there exists a bounded linear mapping $ u: A \rightarrow  X$ such that
\begin{equation}
\lim_{h\rightarrow 0} \frac{\Vert f(p+h)- f(p)- u(h)\Vert_X}{\Vert h \Vert_A}=0,
\end{equation}
and in this case, we denote $u$ by $Df(p)$. Let $D_{p}(A ,X)$ denotes the set of all
 continuous functions $f : A \rightarrow X$, which are
Fr\'{e}chet differentiable on a open neighborhood (say $U$) of $ p $.
The main purpose of differential calculus consists
in getting some information using an affine approximation to a given nonlinear
map around a given point. In many applications it is important to have Fr\'echet derivatives of $f$,
since they provide genuine local linear approximation to $f$. For instance
let $U$ be an open subset of $A$ containing the segment $[x,y]=\{(1-\theta)x+\theta y: 0\leq \theta\leq 1\}$,
and let $f :A\rightarrow X$ be Fr\'echet differentiable on $U$, then
the following mean value formula holds
\begin{align}
\|f(x)-f(y)\leq \|x-y\|\sup_{0< \theta< 1}\left\|Df((1-\theta)x+\theta y)\right\|.
\end{align}

For two Lebesgue integrable functions $f,g:[a,b]\rightarrow \mathbb{R}$, consider the
\v{C}eby$\breve{s}$ev functional:

\begin{equation*}
T(f,g):=\frac{1}{b-a}\int_{a}^{b}f(t)g(t)dt-\frac{1}{b-a}%
\int_{a}^{b}f(t)dt\frac{1}{b-a}\int_{a}^{b}g(t)dt.
\end{equation*}%

In 1934, G. Gr\"{u}ss \cite{gru} showed that
\begin{equation}\label{1.1}
\left\vert T(f,g)\right\vert \leq \frac{1}{4}%
(M-m)(N-n),
\end{equation}%
provided $m,M,n,N$ are real numbers with the property $-\infty <m\leq f\leq
M<\infty $ and $-\infty <n\leq g\leq N<\infty \quad \text{a.e. on }[a,b].$
The constant $\frac{1}{4}$ is best possible in the sense that it cannot be
replaced by a smaller quantity and is achieved for
\[
f(x)=g(x)=sgn \Big(x-\frac{a+b}{2}\Big).
\]
The discrete version of (\ref{1.1}) states that:
If $a\leq a_i\leq A,~b\leq b_i\leq B,~(i=1,...,n)$ where $a,A,b,B,a_i,b_i$ are real
numbers, then
\begin{equation}\label{1.2}
\left| \frac{1}{n}\sum_{i=1}^na_ib_i-\frac{1}{n}\sum_{i=1}^na_i.\frac{1}{n}\sum_{i=1}^nb_i\right|\leq \frac{1}{4}(A-a)(B-b),
\end{equation}
where the constant $\frac{1}{4}$ is the best possible for an arbitrary $n\geq 1$. Some refinements of the
discrete version of Gr\"{u}ss inequality (\ref{1.2}) for inner product spaces are available in \cite{dra,kech}.

\begin{theorem}
Let $(H; \langle \cdot, \cdot\rangle)$ and $\mathbb{K}$ be as above and
$\overline{x}=(x_1, . . . , x_n)\in H^n$, $\overline{\alpha}=(\alpha_1, . . . , \alpha_n) \in \mathbb{K}^n$ and
$\overline{p}=(p_1,...,p_n)$ a probability vector. If $x,X \in H$ are
such that
\begin{equation*}
 Re \left<X- x_i, x_i - x \right>\geq 0~for~ all~ i\in \{1, . . . , n\},
\end{equation*}
or, equivalently,
\begin{equation*}
\left\|x_i-\frac{x+X}{2}\right\|\leq\frac{1}{2}\|X-x\| ~for~ all~ i\in \{1, . . . , n\},
\end{equation*}
holds, then the following inequality holds
\begin{align}\label{1.4}
\left\|\sum_{i=1}^np_i\alpha_i x_i-\sum_{i=1}^n p_i \alpha_i \sum_{i=1}^n p_ix_i\right\|
&\leq\frac{1}{2}\|X-x\|\sum_{i=1}^np_i\left|\alpha_i-\sum_{j=1}^np_j\alpha_j\right|\notag\\
&\leq\frac{1}{2}\|X-x\|\left[\sum_{i=1}^np_i|\alpha_i|^2-\left|\sum_{i=1}^np_i\alpha_i\right|^2\right]^{\frac{1}{2}}.
\end{align}
The constant $\frac{1}{2}$ in the first and second inequalities is best possible.
\end{theorem}

In recent years several refinements and generalizations have been considered for
the Gr\"uss inequality. We would like to refer the reader to \cite{dra1, gha, gru, ili, kech, li, mit}
and references therein for more information.

In this paper, for every Hilbert $C^*$-module $X$ over a $C^*$-algebra $A$, some Gr\"{u}ss type inequalities in semi-inner product $C^*$-module $D_{p}(A ,X^n)$ are established.
 We also for two arbitrary Banach $*$-algebras, define a norm and an involution map on $D_{p}(A ,B)$ and prove that $D_p(A,B)$ is a Banach $*$-algebra.


\section{Gr\"{u}ss type inequalities for differentiable mappings}

Let $A$ be a $C^*$-algebra. A semi-inner product module
over $A$ is a right module $X$ over $A$ together with a generalized semi-inner product,
that is with a mapping $\langle.,.\rangle$ on $X\times X$, which is $A$-valued and has the following properties:
\begin{enumerate}
\item[(i)]
 $\langle x,y+z\rangle=\langle x,y\rangle+\langle x,z\rangle$ for all $x,y,z\in X,$
\item[(ii)]
$\left\langle x,ya\right\rangle=\left\langle x,y\right\rangle a$ for $x,y\in X, a\in A$,
\item[(iii)]
$\langle x,y\rangle ^*=\langle y,x\rangle$ for all $x,y\in X$,
\item[(iv)]
$\left\langle x,x\right\rangle \geq 0$ for $x\in X$.
\end{enumerate}
We will say that $X$ is a semi-inner product $C^*$-module.
If, in addition,
\begin{enumerate}
\item[(v)]
$\langle x,x\rangle=0$ implies $x=0$,
\end{enumerate}
then $\langle .,.\rangle$ is called a generalized inner
product and $X$ is called an inner product module over $A$ or an inner product $C^*$-module.
An inner product $C^*$-module which is complete with respect to its norm $\|x\|=\|\langle x,x\rangle\|^\frac{1}{2}$,
is called a Hilbert $C^*$-module.

As we can see, an inner product module obeys the same axioms as an ordinary
inner product space, except that the inner product takes values in a more general
structure rather than in the field of complex numbers.

If $A$ is a $C^*$-algebra and $X$ is a semi-inner product $A$-module, then the following Schwarz
inequality holds:
\begin{equation}\label{2.1}
\langle x,y\rangle\langle y,x\rangle\leq \|\langle x,x\rangle\|\langle y,y\rangle~(x,y\in X)
\end {equation}
($e.g.$ \cite[Proposition 1.1]{lan}).

\begin{theorem}\label{t2}\cite{gha}
Let $A$ be a  $C^{*}$- Algebra, $X$ a Hilbert  $C^{*}$- module. If $x, y, e \in X$,
$\langle e , e \rangle$ is an idempotent in $A$ and $\alpha, \beta, \lambda, \mu$ are complex numbers such that
\begin{equation*}
\left\Vert x- \frac{\alpha + \beta}{2}e\right\Vert \leq \frac{1}{2}\vert \alpha - \beta \vert,\hspace{1 cm}
\left\Vert y- \frac{\lambda + \mu}{2}e\right\Vert\leq \frac{1}{2}\vert \lambda - \mu \vert,
\end{equation*}
hold, then one has the following inequality;
\begin{equation*}
\left\Vert \langle x ,y\rangle - \langle x , e \rangle \langle e ,y \rangle\right\Vert \leq \frac{1}{4} \mid \alpha - \beta \vert \vert \lambda - \mu \mid.
\end{equation*}
\end{theorem}

\begin{example}
Let $A$ be a real $C^*$-algebra and $X$ be a semi-inner product $C^{*}$-module on a $C^{*}$-algebra $B$.
If functions $f, g \in D_{p}(A ,X)$,
then function $k: A \rightarrow B$ as $k(a)=\langle f(a) , g(a)\rangle$ is differentiable
in $(p\in A)$ and derivative of that is a linear mapping $Dk(p): A \rightarrow B$ defined by
\begin{equation*}
 Dk(p)(a) = \langle Df(p)(a), g(p)\rangle + \langle f(p) , Dg(p)(a) \rangle.
\end{equation*}
Because
\begin{align*}
&\langle f(p+h), g(p+h)\rangle-\langle f(p), g(p)\rangle-\langle Df(p)(h), g(p)\rangle-\langle f(p), Dg(p)(h)\rangle\\
&=\langle f(p+h), g(p+h)-g(p)-Dg(p)(h))\rangle+\langle f(p+h)-f(p), Dg(p)(h)\rangle\\
&+\langle f(p+h)-f(p)-Df(p)(h), g(p)\rangle.
\end{align*}
\end{example}

Let $A$ be a $C^*$-algebra and $X$ a semi-inner product $A$-module. If $f\in D_p(A,X)$ and $a\in A$, we define
the function $f_a: A\rightarrow X$ by $f_{a}(t)=f(t)a$.

\begin{theorem}
Let $X$ be a semi-inner product $C^{*}$-module on $C^{*}$-algebra $A$, and $p\in A, e\in X$.
If $\langle e,e\rangle$ is an idempotent element in $A$, and $f, g \in D_{p}(A,X)$,
then for every $a\in A$, the map $[\cdot,\cdot]_a: D_{p}(A,X) \times D_{p}(A,X) \rightarrow A$ with;
\begin{equation*}
[f,g]_a:=  \big\langle D f(p)(a), Dg(p)(a)\big\rangle_{1} +  \big\langle f(p), g(p)\big\rangle_{1} -D\big\langle f(\cdot), g(\cdot)\big\rangle_{1}(p)(a),
\end{equation*}
is a generalized semi-inner product on $D_{p}(A,X)$, where
 $$\langle f(a),g(a)\rangle_1=\langle f(a),g(a)\rangle-\langle f(a),e\rangle\langle e,g(a)\rangle.$$
\end{theorem}

\begin{proof}
First, we show that $f_a\in D_p(A,X)$ and $Df_{a}(p)=(Df(p))a$.
There exists a bounded convex set $V(=B(p,r))$ containing $p$ such that $V\subseteq U$.
 Let $ p, h\in V, a\in A $, then
\begin{align*}
\Vert f_{a}(p+h)- f_{a}(p)- (Df(p)(h))a\Vert&=\|[f(p+h)- f(p)- Df(p)(h)]a\|\\
&\leq \|f(p+h)- f(p)- Df(p)(h)\|\|a\|.
\end{align*}
This implies that $f_a\in D_p(A,X)$.

A simple calculation shows
\begin{multline*}
[f,g]_a= \Big\langle  Df(p)(a)-f(p) ,Dg(p)(a)-g(p) \rangle \\
- \langle  Df(p)(a)-f(p) ,e\rangle   \langle  e ,Dg(p)(a)-g(p)\Big\rangle\\
=\Big \langle( Df(p)(a)-f(p))-e\big\langle  e, ( Df(p)(a)-f(p))\big\rangle  \\
 , (Dg(p)(a)-g(p))-e \big\langle  e,(Dg(p)(a)-g(p))\big\rangle\Big\rangle.
\end{multline*}
Therefore,
\begin{multline*}
[f,f]_a=\Big \langle( Df(p)(a)-f(p))-e\big\langle  e, ( Df(p)(a)-f(p))\big\rangle  \\
 , (Df(p)(a)-f(p))-e \big\langle  e,(Df(p)(a)-f(p))\big\rangle\Big\rangle\geq0.
\end{multline*}
It is easy to show that $[\cdot , \cdot]_a $ is a generalized semi-inner product on $ D_{p}(A,X)  $.
\end{proof}

\begin{lemma}
Let $X$ be a semi-inner product $C^{*}$-module on $C^{*}$-algebra $A$, and $p,a\in A, e\in X$.
If $\langle e,e\rangle$ is an idempotent element in $A$, $f, g \in D_{p}(A,X)$, and \\
$\alpha, \beta, \alpha', \beta', \mu, \lambda, \mu', \lambda'$ are complex numbers such that
\begin{align*}
&\left\Vert f(p)- \frac{\alpha + \beta}{2}e\right\Vert \leq \frac{1}{2}\vert \alpha - \beta \vert \\
&\left\Vert Df(p)(a)- \frac{\alpha' + \beta'}{2}e\right\Vert \leq \frac{1}{2}\vert \alpha' - \beta' \vert \\
&\left\Vert g(p)- \frac{\lambda + \mu}{2}e\right\Vert\leq \frac{1}{2}\vert \lambda - \mu \vert\\
&\left\Vert Dg(p)(a)- \frac{\mu' + \lambda'}{2}e\right\Vert \leq \frac{1}{2}\vert \mu' - \lambda' \vert,
\end{align*}
then the following inequality holds
\begin{multline*}
\Vert \langle D f(p)(a), Dg(p)(a)\rangle_{1} +  \langle f(p), g(p)\rangle_{1} -D\big\langle f(\cdot), g(\cdot)\big\rangle_{1}(p)(a) \Vert \\
\leq \frac{1}{2}( \vert \alpha - \beta \vert+\vert \alpha' - \beta' \vert)(\vert \lambda - \mu \vert +\vert \lambda' - \mu' \vert ).\\ \\
\end{multline*}
\end{lemma}

\begin{proof}
Since $[\cdot , \cdot]_a $ is a generalized semi-inner product on $ D_{p}(A,X)  $, the Schwartz inequality holds, i.e,\\ \\
\begin{equation*}
\Vert [f,g]_a\Vert^{2} \leq \Vert [f,f]_a\Vert ~~ \Vert [g,g]_a\Vert.
\end{equation*}\\ \\
We know that
\begin{multline*}
\Vert [f , f]_a\Vert \leq \big \Vert  \langle Df(p)(a) , Df(p)(a)\rangle -\langle Df(p)(a) , e \rangle \langle e , Df(p)(a) \rangle \big \Vert
\\+ \big \Vert \langle f(p) , f(p)\rangle -\langle f(p) , e \rangle \langle e , f(p) \rangle \big \Vert
\\ +\big \Vert \langle Df(p)(a) , f(p)\rangle -\langle Df(p)(a) , e \rangle \langle e , f(p) \rangle \big \Vert \\
 +\big \Vert \langle f(p) , Df(p)(a)\rangle -\langle f(p) , e \rangle \langle e , Df(p)(a) \rangle \big \Vert.
\end{multline*}
This inequality and Theorem \ref{t2} imply that
\begin{multline*}
\Vert [f , f]_a\Vert \leq \frac{1}{4} \vert \alpha' - \beta' \vert^{2} + \frac{1}{4}\vert \alpha - \beta \vert^{2}
+ \frac{1}{2} \vert \alpha' - \beta' \vert \vert \alpha - \beta \vert\\
=\frac{1}{4}( \vert \alpha - \beta \vert+\vert \alpha' - \beta' \vert)^{2}.
\end{multline*}
Similarly
\begin{multline*}
\Vert [g , g]_a\Vert\leq \frac{1}{4} \vert \lambda' - \mu' \vert^{2} + \frac{1}{4}\vert \lambda- \mu \vert^{2}
+ \frac{1}{2} \vert \lambda' - \mu' \vert \vert \lambda - \mu \vert \\
=\frac{1}{4}( \vert \lambda - \mu \vert+\vert \lambda' - \mu' \vert)^{2}.
\end{multline*}
\end{proof}

Let $X$ be a semi-inner product $C^{\ast}$-module over $C^*$-algebra $A$. For every $x\in X$,
we define the map $\hat{x}: A\rightarrow X^n$ by $\hat{x}(a)=(xa,...,xa),~~~~(a\in A)$.
\begin{lemma}
Let $X$ be a semi-inner product $C^{\ast}$-module, $x_0,y_0, x_1,y_1\in X$ and $(r_{1},r_{2},....,r_{n})\in \mathbb{R}^n$ a probability vector.
If $p\in A$ and  $f=(f_1,...,f_n), g=(g_1,...,g_n)\in D_{p}(A,X^n)$ such that
\begin{equation*}
\left\|Df(p)-\widehat{\frac{x_0+y_0}{2}}\right\|\leq \left\|\frac{x_0-y_0}{2}\right\|,
\end{equation*}
and
\begin{equation*}
\left\|Dg(p)-\widehat{\frac{x_1+y_1}{2}}\right\|\leq \left\|\frac{x_1-y_1}{2}\right\|,
\end{equation*}
then for all $a\in A$, we have

\begin{multline}
\left\|\sum_{i=1}^{n}r_{i}\left\langle Df_{i}(p)(a), Dg_{i}(p)(a)\right\rangle
-\Big\langle \sum_{i=1}^{n}r_{i}Df_{i}(p)(a) ,\sum_{i=1}^{n}r_{i}Dg_{i}(p)(a)\Big\rangle\right\|\\
\leq \frac{1}{4}\|x_0-y_0\|\|x_1-y_1\|\|a\|^2.
\end{multline}
\end{lemma}

\begin{proof}
For every $a\in A$, we define the map $\big(\cdot,\cdot\big)_a: D_{p}(A,X^n) \times D_{p}(A,X^n) \rightarrow A$ with;
\begin{multline*}
\big(f,g\big)_a= \sum_{i=1}^{n}r_{i}\Big\langle Df_{i}(p)(a), Dg_{i}(p)(a)\Big\rangle
-\Big\langle \sum_{i=1}^{n}r_{i}Df_{i}(p)(a) ,\sum_{i=1}^{n}r_{i}Dg_{i}(p)(a)\Big\rangle.
\end{multline*}
The following Korkine type inequality for differentiable mappings holds:
\begin{equation*}
\big(f,g\big)_a= \frac{1}{2}\sum_{i=1,j=1}^{n}r_{i}r_{j}\Big\langle Df_{i}(p)(a)-Df_{j}(p)(a),Dg_{i}(p)(a)-Dg_{j}(p)(a)\Big\rangle,
\end{equation*}
Therefore, $\big(f,f\big)_a\geq 0$. It is easy to show that $\big(\cdot,\cdot\big)_a$ is a generalized semi-inner product on $ D_{p}(A,X^n)$.

A simple calculation shows that
\begin{multline*}
\big(f,g\big)_a= \sum_{i=1}^{n}r_{i}\Big\langle Df_{i}(p)(a)-\frac{x_0+y_0}{2}a,~ Dg_{i}(p)(a)-\frac{x_1+y_1}{2}a\Big\rangle\\
-\Big\langle \sum_{i=1}^{n}r_{i}Df_{i}(p)(a)-\frac{x_0+y_0}{2}a ,~\sum_{i=1}^{n}r_{i}Dg_{i}(p)(a)-\frac{x_1+y_1}{2}a\Big\rangle.
\end{multline*}

From Schwartz inequality, we have
\begin{align*}
\|\big(f,g\big)_a\|^2&\leq\sum_{i=1}^{n}r_{i}\left\| Df_{i}(p)(a)-\frac{x_0+y_0}{2}a\right\|^2\sum_{i=1}^{n}r_{i}\left\| Dg_{i}(p)(a)-\frac{x_1+y_1}{2}a\right\|^2\\
&\leq \left\|Df(p)-\widehat{\frac{x_0+y_0}{2}}\right\|^2\left\|Dg(p)-\widehat{\frac{x_1+y_1}{2}}\right\|^2\|a\|^4\\
&\leq \frac{1}{16}\|x_0-y_0\|^2\|x_1-y_1\|^2\|a\|^4
\end{align*}
\end{proof}

\begin{corollary}
Let $X$ be a semi-inner product $C^{\ast}$-module, $x_0,y_0\in X, (\alpha_1, . . . , \alpha_n)
 \in \mathbb{C}^n$ and $(r_{1},r_{2},....,r_{n})\in \mathbb{R}^n$ a probability vector. If $p\in A$ and $f=(f_1,...,f_n)\in D_{p}(A,X^n)$ such that
\begin{equation*}
\left\|Df(p)-\widehat{\frac{x_0+y_0}{2}}\right\|\leq \left\|\frac{x_0-y_0}{2}\right\|,
\end{equation*}
then for all $a\in A$, we have
\begin{align}\label{4.2}
&\left\|\sum_{i=1}^nr_i\alpha_iDf_{i}(p)(a)-\sum_{i=1}^n r_i \alpha_i \sum_{i=1}^n r_iDf_{i}(p)(a)\right\|\notag\\
&\leq \|a\|\left\|\frac{x_0-y_0}{2}\right\|\left[\sum_{i=1}^nr_i|\alpha_i|^2-\left|\sum_{i=1}^nr_i\alpha_i\right|^2\right]^{\frac{1}{2}}.
\end{align}
\end{corollary}

\begin{proof}
\begin{align*}
&\left\|\sum_{i=1}^nr_i\alpha_iDf_{i}(p)(a) -\sum_{i=1}^n r_i \alpha_i \sum_{i=1}^n r_iDf_{i}(p)(a)\right\|\\
&=\left|\sum_{i=1}^nr_i\Big(\alpha_i-\sum_{j=1}^nr_j\alpha_j\Big)\right|\left\|Df_{i}(p)(a)-\frac{x_0+y_0}{2}.a\right\|\\
&\leq\sum_{i=1}^nr_i\left|\alpha_i-\sum_{j=1}^nr_j\alpha_j\right|\left\|Df(p)-\widehat{\frac{x_0+y_0}{2}}\right\|\|a\|\\
&\leq \|a\|\left\|\frac{x_0-y_0}{2}\right\|\left[\sum_{i=1}^nr_i|\alpha_i|^2-\left|\sum_{i=1}^nr_i\alpha_i\right|^2\right]^{\frac{1}{2}}.
\end{align*}
\end{proof}

\begin{corollary}
Let $X$ be a semi-inner product $C^*$-module, $x_0,y_0\in X$. If $p\in A$ and $f=(f_1,...,f_n)\in D_{p}(A,X^n)$ such that
\begin{equation*}
\left\|Df(p)-\widehat{\frac{x_0+y_0}{2}}\right\|\leq \left\|\frac{x_0-y_0}{2}\right\|,
\end{equation*}
then for all $a\in A$, we have
\begin{equation}\label{4.3}
\left\|\sum_{k=1}^nkDf_{k}(p)(a)-\frac{n+1}{2}.\sum_{k=1}^n Df_{k}(p)(a)\right\|\leq \frac{\|a\|\|x_0-y_0\|n}{4} \left[\frac{(n-1)(n+1)}{3}\right]^{\frac{1}{2}},
\end{equation}
and
\begin{multline}\label{4.4}
\left\|\sum_{k=1}^nk^2Df_{k}(p)(a)-\frac{(n+1)(2n+1)}{6}.\sum_{k=1}^n Df_{k}(p)(a)\right\|\\
\leq\frac{\|a\|\|x_0-y_0\|n}{12\sqrt{5}}\sqrt{(n-1)(n+1)(2n+1)(8n+11)}.
\end{multline}
\end{corollary}

\begin{proof}
If we put $ri=\frac{1}{n}, \alpha_i=k$ in inequality \eqref{4.2}, then we get \eqref{4.3},
and if $ri=\frac{1}{n}, \alpha_i=k^2$ in inequality \eqref{4.2}, then we get \eqref{4.4}.
\end{proof}

\section{Differential mapping on Banach *-algebras}

\begin{theorem}
Let $A,B$ be two Banach $*$-algebras and $p$ be a self adjoint element in $A$.
Then $D_p(A,B)$ is a Banach $*$-algebra with the point-wise operations and the involution
$f^*(a)=(f(a^*))^*~~(a\in A)$, and the norm
\begin{equation}\label{3.1}
\|f\|:=\max\left\{\sup_{x\in U}\|Df(x)\|,~\sup_{a\in A}\|f(a)\|\right\}< \infty.
\end{equation}

\end{theorem}

\begin{proof}
First we show that the involution $f\rightarrow f^*$ is differentiable and $ Df^*(x)(h)=(Df(x^*)(h^*))^*~~~~~(x,h\in U=B(p,r))$.
It is trivial that if $x\in U$, then $x^*\in U$ because $\|x-p\|=\|(x-p)^*\|=\|x^*-p\|$.
It can be shown easily that $Df^*(x)$ is a bounded linear map with $\|Df^*(x)\|=\|Df(x^*)\|$. Therefore
\begin{align*}
\|f^*(x+h)&-f^*(x)-Df^*(x)(h)\|=\|\big(f(x^*+h^*)-f(x^*)-Df(x^*)(h^*)\big)^*\|\\
&=\|f(x^*+h^*)-f(x^*)-Df(x^*)(h^*)\|\\
&\leq \epsilon \|h^*\|= \epsilon \|h\|.
\end{align*}
From $\|Df^*(x)\|=\|Df(x^*)\|$ and $\|f^*(a)\|=\|f(a^*)\|$, we obtain
\begin{align*}
\|f^*\|&=\max\left\{\sup_{x\in U}\|Df^*(x)\|,~\sup_{a\in A}\|f^*(a)\|\right\}\\
&=\max\left\{\sup_{x\in U}\|Df(x^*)\|,~\sup_{a\in A}\|f(a^*)\|\right\}\\
&=\max\left\{\sup_{x^*\in U}\|Df(x^*)\|,~\sup_{a^*\in A}\|f(a^*)\|\right\}=\|f\|.
\end{align*}
Now, we show that $D_p(A,B)$ is complete.
There exists a bounded convex set $V(=B(p,r))$ containing $p$ such that $V\subseteq U$.
Suppose that $(f_n)$ is a Cauchy sequence in $D_{p}(A ,B)$, i.e.,
$$\|f_n(a)-f_m(a)\|\rightarrow 0~~(a\in A), \text{ and } \|Df_n(x)-Df_m(x)\|\rightarrow 0~~(x\in V).$$

Since $B$ is complete, therefore $L(A,B)$ the space of all bounded linear maps from $A$ into $B$, is complete.
 So there are functions $f,g$ such that $\sup_{a\in A}\|f_n(a)-f(a)\|\rightarrow 0$
and $\sup_{x\in V}\|Df_n(x)-g(x)\|\rightarrow 0$.
Given $\varepsilon > 0$, we can find $N\in \mathbb{N}$ such that for $m > n \geq N$ one has
\begin{align}
\|Df_m-Df_n\|_{\infty}&=\sup_{x\in V}\|Df_m(x)-Df_n(x)\|<\frac{\varepsilon}{3}\label{3.2}\\
\|g-Df_n\|_{\infty}&=\sup_{x\in V}\|g(x)-Df_n(x)\|<\frac{\varepsilon}{3}.\label{3.3}
\end{align}

We may suppose that there exist $a\in A$ such that, $p+a\in V$. Using Lipschitzian functions $f_m-f_n$, we obtain that

\begin{align*}
\|f_m(p+a)-f_m(p)&-(f_n(p+a)-f_n(p))\|\\
&\leq\sup_{0<\theta<1}\|Df_m(p+\theta a)-Df_n(p+\theta a)\|\|a\|\leq \frac{\varepsilon}{3}\|a\|.
\end{align*}

passing to the limit on $m$, we get
\begin{align}\label{3.4}
\|f(p+a)-f(p)-(f_n(p+a)-f_n(p))\|\leq \frac{\varepsilon}{3}\|a\|.
\end{align}
Utilizing differentiability $f_N$ and \eqref{3.3}, we have
\begin{align}\label{3.5}
\|f_N(p+a)-f_N(p)-g(p)(a)\|&\leq\|f_N(p+a)-f_N(p)-Df_N(p)(a)\|\notag\\
&+\|Df_N(p)(a)-g(p)(a)\|\leq \frac{\varepsilon}{3}\|a\|+\frac{\varepsilon}{3}\|a\|.
\end{align}
From \eqref{3.4} and \eqref{3.5}, we obtain
\begin{equation*}
\|f(p+a)-f(p)-g(p)(a)\|\leq \varepsilon\|a\|.
\end{equation*}
Therefore $D_p(A,B)$ is a Banach $*$-algebra.
\end{proof}



\begin{thebibliography}{99}


\bibitem{dra} S. S. Dragomir, \emph{Advances in Inequalities of the Schwarz,
Gr\"{u}ss and Bessel Type in Inner Product Spaces}, Nova Science puplishers
Inc., New York, 2005.

\bibitem{dra1} S. S. Dragomir, \emph{A Gr\"{u}ss type discrete inequality in inner product spaces
and applications}, J. Math. Anal. Appl., \textbf{250} (2000), 494-511.

\bibitem{gha} A.G. Ghazanfari, S.S. Dragomir, \emph{Bessel and Gr\"{u}ss type inequalities in inner product modules over Banach ${*}$-algebra},
Linear Algebra Appl. \textbf{434} (2011), 944-956.

\bibitem{gru} G. Gr\"{u}ss, \emph{\"{U}ber das Maximum des absoluten Betrages
von $\frac{1}{b-a} \int_a^b f(x)g(x)dx-\frac{1}{(b-a)^2}\int_a^bf(x)dx%
\int_a^bg(x)dx $}, Math. Z. \textbf{39}(1934), 215-226.

\bibitem{ili} D. Ili\v{s}evi\'c and S. Varo\v{s}anec, \emph{Gr\"uss type inequalities in inner product modules},
Proc. Amer. Math. Soc. \textbf{133} (2005), 3271-3280.

\bibitem{kech} A. I. Kechriniotis and K. K. Delibasis, \emph{On generalizations of Gr\"{u}ss inequality in
inner Product Spaces and applications}, J. Inequal. Appl. Vol(2010), Article ID 167091.

\bibitem{lan} E.C. Lance, \emph{Hilbert $C^*$-Modules}, London Math. Soc. Lecture Note Series \textbf{210}, Cambridge
Univ. Press, 1995.

\bibitem{li} X. Li, R. N. Mohapatra and R. S. Rodriguez, \emph{Gr\"uss-type inequalities}, J. Math. Anal. Appl.
\textbf{267} (2002), no. 2, 434-443.


\bibitem{mit} D. S. Mitrinovi\'c, J. E. Pe\v{c}ari\'c, and A. M. Fink, \emph{Classical and New Inequalities in Analysis},
Kluwer Academic, Dordrecht, 1993.

\end{thebibliography}
\end{document}